\documentclass[12pt,bezier]{article}
\usepackage{amssymb}
\usepackage{mathrsfs}
\usepackage{amsmath}
\usepackage{amsfonts,amsthm,amssymb}
\usepackage{amsfonts}
\usepackage{graphics}
\usepackage{CJK}

\newtheorem{lem}{Lemma}[section]
\newtheorem{thm}[lem]{Theorem}

\newtheorem{exam}[lem]{Example}

\begin{document}

\title{On the number of proper paths between vertices in edge-colored hypercubes}

\author{ Lina Xue, \quad Weihua Yang\footnote{Corresponding author. E-mail: ywh222@163.com,~yangweihua@tyut.edu.cn},\quad Shurong Zhang\\
\\ \small Department of Mathematics, Taiyuan University of
Technology,\\
\small  Taiyuan Shanxi-030024,
China}
\date{}
\maketitle

{\small{\bf Abstract:} Given an integer  $1\leq j <n$, define the
$(j)$-coloring of a $n$-dimensional hypercube $H_{n}$ to be the
$2$-coloring of the edges of $H_{n}$ in which all edges in dimension
$i$, $1\leq i \leq j$, have color $1$ and all other edges have color
$2$. Cheng et al.
  [Proper distance in edge-colored hypercubes, Applied Mathematics and Computation 313 (2017) 384-391.]
   determined the number of distinct shortest properly
colored paths between a pair of vertices for the $(1)$-colored
hypercubes. It is natural to consider the number for $(j)$-coloring, $j\geq 2$. In this note, we determine the number of different shortest
proper paths in $(j)$-colored hypercubes for arbitrary $j$.


\section{Introduction}

In this work, all graph colorings are colorings of the edges. A coloring of a
graph (or subgraph) is called $proper$ if no two adjacent edges in the
graph (respectively subgraph) have the same color. A colored graph
is called $properly$ $connected$ if there is a properly colored path
between every pair of vertices \cite{Eddie Cheng}.

Given a properly connected coloring of a graph, Coll et al. \cite{Coll} initiated
the study of the proper distance between pairs
of vertices in the hypercube. The $proper$ $distance$ between two vertices $u$ and $v$,
denoted by $pd(u,v)$, is the minimum length (number of edges) of a
properly colored path between $u$ and $v$. Although the proper
connection number has been studied and applications to computer
science and networking have been described, precious little
attention has been paid to the proper distance which is, in some
sense, a measure of the efficiency of the properly connected
network \cite{Eddie Cheng}.

Numerous works related to proper connectivity and the proper
connection number have been studied in recent years, see the dynamic
survey \cite{Li} for more details. The study of proper connection in graphs is motivated
in part by applications in networking and network security. In addition to security,
of course it is critical that message transmission in a network be fast \cite{Eddie Cheng}.
So we consider the number of different shortest proper colored paths between
any pair of vertices in $(j)$-colored hypercubes.

An $n$-dimensional hypercube $H_{n}$ (with $n\geq2$) is an undirected
graph $H_{n}=(V,E)$ with $|V|=2^{n}$ and $|E|=n2^{n-1}$. Every vertex
can be represented by an $n$-bit binary string. There is an edge
between two vertices whenever their binary string representation
differs in only one bit position. The $i$-dimensional edges of
hypercubes is the set of edges whose two ends have the different
$i^{th}$ bit position. Given $1\leq j <n$, define the $(j)$-coloring
of $H_{n}$ to be the $2$-coloring of the edges of $H_{n}$ in which
all edges in dimension $i$, where $1\leq i \leq j$, have color $1$ and all
other edges have color $2$. For example, the $(1)$-coloring of
$H_{n}$ has a perfect matching in color $1$ and all other edges in
color $2$.

For any two vertices $u$ and $v$ of $H_n$, let $o(u,v)$ denote the
number of positions $i$ for $1\leq i\leq j$ in which $u$ and $v$
differ. Let $t(u,v)$ denote the number of positions $i$ for $j<
i\leq n$ in which $u$ and $v$ differ. Let $\gamma(u,v)$ be the
indicator variable which takes the value $1$ if $o(u,v)+t(u,v)$ is
odd and 0 if $o(u,v)+t(u,v)$ is even.

Cheng et al. proved the proper distance between any two
vertices of a $(j)$-colored hypercube.
 \begin{thm}[\cite{Eddie Cheng}]
 Let $n\geq2$, $1\leq j < n$, and $H$ be the $(j)$-coloring of $H_{n}$. Then $H$ is properly connected. Furthermore, let $u$ and $v$ be any pair of vertices in $H$. Then
 $$pd(u,v)=2max(o(u,v),t(u,v))-\gamma(u,v)$$
\end{thm}

 For the number of the shortest properly colored paths in edge-colored hypercubes, there are no general results.
 Cheng et al. determined the special case $j=1$ in $(j)$-colored hypercubes \cite{Eddie
 Cheng}. Let $pp(u,v)$ denote the number of distinct shortest properly colored paths in $H$ from $u$ to $v$ and
  $a(m)$ is the number of all possible permutations and combinations after $m$  transformations of the first $j$ bits or the last $n-j$ bits.
 If we consider the first $j$ bits, let $a(m)=a_{o}(m)$; Otherwise, 
 let $a(m)=a_{t}(m)$.

\begin{thm}[\cite{Eddie Cheng}]
Let $n\geq2$ and let $H$ be the $(1)$-coloring of $H_n$. Then $H$ is properly connected. Furthermore, let $u$ and $v$ be any pair of distinct vertices whose strings differ in $t = t(u, v )$ positions other than the first. If $t=0$, then $u$ and $v$ are adjacent so
$pd(u, v ) = 1$ and $pp(u, v ) = 1$. Otherwise, for $t \geq 1$, we have $pd(u, v ) = 2t-\gamma(u, v )$ and $pp_H(u, v ) = (2-\gamma(u,v ))(t!)$.
\end{thm}

 In next section, we determine the case for $j>1$.
 Moreover, a more general result will be discussed. For the structural properties we refer to \cite{liyang,Harary,Yang 09, Yang 10,Yang 12}.

 \section{Number of Proper Paths Between Vertices}
In this section, we shall count the number of different shortest proper paths
between any two vertices in a $(j)$-colored hypercube. Cheng et al. induced a method to find properly colored paths between two given vertices $u$ to $v$. We will follow the method in our proof to find shortest properly colored paths between two given vertices. We give an  example as follows:

 \begin{exam}[\cite{Eddie Cheng}]
If $n=7$ and $j=4$ and we consider $u$=$( 0 , 1 , 0 , 1 , 0 , 0 , 0 )$, and $v$=$( 0 , 0 , 1 , 1 , 1 , 1 , 1 )$, then some examples of
proper colored paths from $u$ to $v$ created by the algorithm above would be as follows:

\begin{equation*} \label{eq:1}
\begin{split}
u =
  (0,1,0,1,0,0,0)\rightarrow\\
  (0,0,0,1,0,0,0)\rightarrow\\
  (0,0,0,1,1,0,0)\rightarrow\\
u^{\prime}=(0,0,1,1,1,0,0)\rightarrow\\
  (0,0,1,1,1,1,0)\rightarrow\\
  (1,0,1,1,1,1,0)\rightarrow\\
  (1,0,1,1,1,1,1)\rightarrow\\
v =(0,0,1,1,1,1,1)
 \end{split}
 \end{equation*}

\end{exam}
This is an arbitrary proper colored path.

\begin{equation*} \label{eq:1}
\begin{split}
u =
(0,1,0,1,0,0,0)\rightarrow\\
(0,1,0,1,1,0,0)\rightarrow\\
(0,0,0,1,1,0,0)\rightarrow\\
(0,0,0,1,1,1,0)\rightarrow\\
u^{\prime}=(0,0,1,1,1,1,0)\rightarrow\\
v =(0,0,1,1,1,1,1)
 \end{split}
 \end{equation*}

This is a shortest proper colored path.

In fact, in the hypercube, any path from $u$ to $v$
consists of  flipping the bits of $u$ until arrive at $v$ \cite{Eddie Cheng}.
A properly colored path from $u$ to $v$  consists of flipping
the bits of $u$ until we arrive at $v$ with the restriction that the
bits we flip must alternate between the first $j$ bits and the last
$n-j$ bits. To construct such a properly colored path, flip the bits
of $u$ that differ with the respective bits of $v$, making sure to
flip the bits in an order that alternates between the first $j$ bits
and the last $n-j$ bits, until we arrive at a vertex $u{'}$ such
that either $o(u{'},v)=0$ or $t(u{'},v)=0$.

If $o(u{'},v)=0$, we alternate flipping the first bit with flipping
the remaining bits in the last $n-j$ positions that differ with the
respective bits of $v$. Otherwise if $t(u{'},v)=0$, then we alternate
flipping the last bit with flipping the remaining bits in the first $j$
positions that differ with the respective bits of $v$. Since there are
a finite number of positions in which $u$ and $v$ differ and the process
will not repeat vertices, this process will eventually terminate, resulting
in a properly colored path from $u$ to $v$ \cite{Eddie
 Cheng}.

Certainly the path contains at least $o(u, v)$ bits flipped from the first $j$ bits and at least $t(u, v)$ bits flipped from the last
$n-j$ bits. A properly colored path, the flipping must alternate between the first $j$ bits and the last $n-j$
bits, there must be $2max(o(u,v),t(u,v))-\gamma(u,v)$ bits flipped in total.
Clearly, if $o=t$, the length of the shortest proper path from $u$ to $v$ is $2o(u,v)-\gamma(u,v)=2o(u,v)=2t(u,v)$, then
we can first flip a bit from the first $j$ bits or the last $n-j$ bits of the $u$. Otherwise, we should consider the size of $o$ and $t$, as the number of the shortest proper colored path is determined
by $pd(u,v)$ and the first flipped bit, where the first flipped bit can only be in the larger side of the $o$ and the $t$.

\begin{thm}\label{thm2.1}
 Let $n\geq2$, $u,v\in V(H_n)$, and let $H$ be the $(j)$-coloring of $H_{n}$.
 Suppose  $o=o(u,v)$  and  $t=t(u,v)$,
then  $pp_{H}(u,v)=(2-\gamma(u,v))(max(o!,t!))a(m)$.
\end{thm}
\begin{proof}
Without loss of generality, let $v=(0,0,\ldots,0)$ and
let $u$ be an arbitrary vertex of $H_{n}-\{v\}$. If $u$ and $v$
are adjacent, then
$pp_{H}(u,v)=(2-\gamma(u,v))(max(o!,t!))a(m)=(2-1)(1!)a(1)=1$. We
then consider $u$ and $v$ are not adjacent.

Note that the number of shortest properly colored paths in the
$(1)$-coloring of $H_{n}$ be given in \cite{Eddie Cheng} and
$pp_{H}(u,v)=(2-\gamma(u,v))(t!)$. So we consider $2\leq j< n$.
Clearly, every proper path from $u$ to $v$ must use alternative
edges in color $1$ and color $2$. According to the method of finding
the shortest proper colored paths in edge-colored hypercubes,
 the first flipped bit of the shortest
proper coloring paths is determined by the size of $o$ and $t$. We then consider the following three cases.

 {\bf Case 1.}\quad $o=t$, $pp_{H}(u,v)=2(o!)(t!)$.

 In this case, since $\gamma(u,v)=0$, there is an even total number of positions in which $u$ and $v$ differ.
 Then every shortest proper path will start on an edge of color $1$ or color $2$,
 by switching $1$ to $0$ in some positions.
 Let $\pi$ be an arbitrary ordering of the $o$ '1's in the first $j$ bits in $u$
 and let $\pi{'}$ be an arbitrary ordering of the $t$ '1's in the last $n-j$ positions in $u$.
 For any choice of $\pi$ and $\pi{'}$, the sequence of switches of positions $\pi(1), \pi{'}(1), \pi(2), \pi{'}(2),\ldots, \pi(t), \pi{'}(t)$ and the
 sequence of switches of positions $ \pi{'}(1), \pi(1), \pi'(2), \pi(2),\ldots, \pi'(t), \pi(t)$ both constitute shortest proper paths from $u$ to $v$ and are
 all distinct. Clearly these paths have length $2t-\gamma(u,v)=2t=2o$.
 Also there can be no other shortest proper path from $u$ to $v$ because every proper path from $u$ to $v$ must include all $t$ of the '1's and  all $o$ of the '1's and must also alternate switching the first $j$ bits and the last $n-j$ bits. Since there are $(o!)(t!)$ such orderings, there are a total of $2(o!)(t!)=(2-\gamma(u,v))(o!)(t!)$ distinct shortest proper paths from $u$ to $v$.

 The proof of $2(o!)(t!)=(2-\gamma(u,v))(max(o!,t!))a(m)$ will be given in later.

 {\bf Case 2.}\quad $o>t$. 

 Since the parity of the $\gamma$ affects the arrangement of the shortest proper coloring paths, we then consider two subcases for $\gamma(u,v)=0$ and $\gamma(u,v)=1$.

 {\bf Subcase 2a.}\quad $\gamma(u,v)=0$, $pp_{H}(u,v)=2(o!)a(m)$.

 In this case, we have $o>t$ and $\gamma(u,v)=0$, this means that every shortest proper path will start on an edge of color $1$ or color $2$.
 As a proper colored path must alternate between the first $j$ bits and the last $n-j$ bits, and the length of the proper colored path from $u$ to $v$ is certain. $pd(u,v)=2\max (o(u,v),t(u,v))-\gamma(u,v)$, we consider all possible permutations of the first $j$ bits.

 In order to distinguish symbols, let $j=l$ , $m=\lceil\frac{1}{2}pd(u,v)\rceil$ and $a(m)=a_{o}(m)$ be all possible permutations of the first $l$ bits. Let $k_{i}$ is the number of the $i$ bit flips from $u$ to $v$ in the first $l$ bits, there are $o$ positions in which $u$ and $v$ differ in the first $l$ bits, so $ k_{1},k_{2},\cdots,k_{o} $ are odd; $k_{j}$ is the number of the $j$ bit flips from $u$ to $v$ in the first $l$ bits, this means that $k_{o+1},k_{o+2},\cdots,k_{l}$ are even, in total $l-o$ positions. Then,
 \begin{equation*} \label{eq:1}
\begin{split}
 a(m)=a_{o}(m)=\sum_{\substack{k_{1}+k_{2}+\cdots+k_{l}=m\\   k_{i}\mbox{ is odd},\ 1\leq i\leq o\\k_{i}\mbox{ is even},\ o+1\leq i\leq l}}\frac{m!}{k_{1}!k_{2}!\cdots k_{l}!}
 \end{split}
 \end{equation*}

By using generating function technique, we derive a formula for $a(m)$ as follows.
\begin{equation*} \label{eq:1}
\begin{split}
A(x) &=
\sum_{m\geq0}a(m)\frac{x^{m}}{m!}
\end{split}
 \end{equation*}
 \begin{equation*} \label{eq:1}
\begin{split}
 &= \sum_{m\geq0}\sum_{\substack{k_{1}+k_{2}+\cdots+k_{l}=m\\   k_{i}\mbox{ is odd},\ 1\leq i\leq o\\k_{i}\mbox{ is even},\ o+1\leq i\leq l}}\frac{m!}{k_{1}!k_{2}!\cdots k_{l}!}\frac{x^{k_{1}+k_{2}+\cdots+k_{l}}}{m!}\\
&=\prod_{i=1}^{o}\big(\sum_{k_{i}\mbox{ is odd}}\frac{x^{k_{i}}}{k_{i}!}\big)^{o}\prod_{i=o+1}^{l}\big(\sum_{k_{i}\mbox{ is even}}\frac{x^{k_{i}}}{k_{i}!}\big)^{l-o}\\
&=\big(\sum_{i\geq0}\frac{x^{2i+1}}{(2i+1)!}\big)^{o}(\sum_{i\geq0}\frac{x^{2i}}{(2i)!})^{l-o}\\
&=\big(\frac{e^{x}-e^{-x}}{2}\big)^{o}\big(\frac{e^{x}+e^{-x}}{2}\big)^{l-o}\\
&=\frac{1}{2^{l}}(e^{-x})^{l}(e^{2x}-1)^{o}(e^{2x}+1)^{l-o}\\
&=\big(\frac{e^{^{-x}}}{2}\big)^{l}\big(\sum\limits_{i=0}^o\binom{o}{i}e^{2ix}(-1)^{o-i}\big)\big(\sum\limits_{j=0}^{l-o} \binom{l-o}{j}e^{2jx}\big)\\
&=(\frac{e^{^{-x}}}{2})^{l}\sum\limits_{i=0}^o (-1)^{o-i}\sum\limits_{j=0}^{l-o} \binom{o}{i}\binom{l-o}{j}e^{2(i+j)x}\\
&=(\frac{1}{2^{l}})\sum\limits_{i=0}^o \sum\limits_{j=0}^{l-o}(-1)^{o-i}\binom{o}{i} \binom{l-o}{j}e^{(2(i+j)-l)x}
 \end{split}
 \end{equation*}

Combining the above arguments, so we have
 \begin{equation*} \label{eq:1}
\begin{split}
 a(m)=a_{o}(m)=\frac{1}{2^{l}}\sum\limits_{i=0}^o \sum\limits_{j=0}^{l-o}(-1)^{o-i} \binom{o}{i} \binom{l-o}{j}(2i+2j-l)^{m}
 \end{split}
 \end{equation*}

So $o>t$ and $\gamma(u,v)=0$, $pp_{H}(u,v)=2(o!)a_{o}(m)=2(o!)a(m)$, the proof is complete.

For case $1$, as $ m=o=t $ and $ k_{i}=1, 1\leq i\leq l$,
\begin{equation*} \label{eq:1}
\begin{split}
 a(m)&=\sum_{\substack{k_{1}+k_{2}+\cdots+k_{l}=m\\ k_{i}\mbox{ is odd}\ k_{i}=1\ 1\leq i\leq o\\k_{i}\mbox{ is even}\ k_{i}=0\ o+1\leq i\leq l}}\frac{m!}{k_{1}!k_{2}!\cdots k_{l}!}\\
 &=\frac{o!}{1!1!\ldots 1!}=o!=t! ,
 \end{split}
 \end{equation*}
we completed $(2-\gamma(u,v))max(o!,t!)a(m)=2(o!)(t!)$.

 {\bf Subcase 2b.}\quad $\gamma(u,v)=1$, $pp_{H}(u,v)=(o!)a(m)$.

  In this case, we have $o>t$ and $\gamma(u,v)=1$, this means that every  shortest proper path will start on an edge of color $1$. As the length of the proper colored path from $u$ to $v$ is certain $pd(u,v)=2\max (o(u,v),t(u,v))-\gamma(u,v)$, and a proper colored path must alternate between the first $j$ bits and the last $n-j$ bits, we consider all possible permutations of the former $j$ bits. By subcase 2a, we have calculated $a(m)$. So if $o>t$ and $\gamma(u,v)=1$, then $pp_{H}(u,v)=(o!)a(m)$.

 {\bf Case 3.}\quad $o<t$. 

In this case, the prove is similar to case $2$, we first consider the last $n-j$ bits (since $o<t$ and $pd(u,v)=2t(u,v)-\gamma(u,v)$ ).
Similarly case $2$, exchange $o$ and $t$,
we have $\gamma(u,v)=0$, $pp_{H}(u,v)=2(t!)a(m)$; $\gamma(u,v)=1$, $pp_{H}(u,v)=(t!)a(m)$,
where
  \begin{equation*} \label{eq:1}
\begin{split}
 a(m)&=a_{t}(m)=\sum_{\substack{k_{1}+k_{2}+\cdots+k_{l}=m\\   k_{i}\mbox{ is odd},\ 1\leq i\leq t\\k_{i}\mbox{ is even},\ t+1\leq i\leq l}}\frac{m!}{k_{1}!k_{2}!\cdots k_{l}!}\\
& =\frac{1}{2^{l}}\sum\limits_{i=0}^t \sum\limits_{j=0}^{l-t}(-1)^{t-i} \binom{t}{i} \binom{l-t}{j}(2i+2j-l)^{m}
 \end{split}
 \end{equation*}
\end{proof}

Let $N=N_1\cup N_2$, where $N=\{1,2,\cdots, n\}$,
$N_1=\{i_1,i_2,\cdots,i_j\}$ and $\{i_{j+1},\cdots,i_n\}$, $i_l\in
N, l=1,2,\cdots,n$. We may naturally refine the  2-edge coloring by
defining two set of edges $E_1$ and $E_2$, where $E_1$ is set of all
$i$-dimensional edge for $i\in N_1$ and $E_2$ is the set of all
$i$-dimensional edges for $i\in N_2$,  such that the  edges in $E_i$
having color $i$. The new definition is called $j^*$-coloring of
hypercubes. In fact, according to the method in the proof of Theorem
2.1, one can see that the Theorem is still true for any
$j^*$-colored hypercubes. 

By the arguments similar to that of Theorem~\ref{thm2.1}, we have the following.

 \begin{thm}
 Let $n\geq2$, and let $H$ be the  $j^*$-coloring of $H_{n}$. Then $H$ is properly connected. Furthermore, let $u$ and $v$ be any pair of distinct vertices whose strings differ both in $o=o(u,v)$ positions for $1\leq i\leq j$ and $t=t(u,v)$ positions for $j< i< n$. We have $pd(u,v)=2max(o,t)-\gamma(u,v)$ and $pp_{H}(u,v)=(2-\gamma(u,v))(max(o!,t!))a(m)$.

\end{thm}

 \section{Conclusions}

 In this paper, we count the number of different shortest proper paths in the $(j)$-colored hypercube, note that the $(j)$-coloring of $H_{n}$ to be the $2$-coloring of the edges of $H_{n}$. We show that given the $(j)$-coloring of $H_{n}$, $pp_{H}(u,v)=(2-\gamma(u,v))(max(o!,t!))a(m)$.

 \section{Acknowledgements}
The research is supported by NSFC (No.11671296, 61502330),  PIT of Shanxi, Research Project Supported by Shanxi Scholarship Council of China, and Fund Program for the Scientific Activities of Selected
Returned Overseas Professionals in Shanxi Province.

\end{document}